\documentclass{amsart}
\usepackage{amssymb}
\usepackage{amscd}
\usepackage{graphicx}
\usepackage{amsmath}
\usepackage{setspace}
\usepackage[all]{xy}
\usepackage{color}

\newtheorem{theorem}{Theorem}[section]

\newtheorem{corollary}[theorem]{Corollary}

\newtheorem{lemma}[theorem]{Lemma}

\newtheorem{proposition}[theorem]{Proposition}

\theoremstyle{remark}
\newtheorem{definition}[theorem]{Definition}

\newtheorem{remark}[theorem]{Remark}

\newcommand{\CA}{\mathcal{A}}
\newcommand{\CB}{\mathcal{B}}

\newcommand{\CALD}{\mathcal{D}}
\newcommand{\CF}{\mathcal{F}}

\newcommand{\CT}{\mathcal{T}}

\newcommand{\CV}{\mathcal{V}}

\newcommand{\rmD}{\mathrm{D}}

\newcommand{\rmK}{\mathrm{K}}

\newcommand{\BBZ}{{\mathbb Z}}

\newcommand{\Hom}{\mathrm{Hom}}
\newcommand{\Ima}{\mathrm{Im}}

\newcommand{\Ker}{\mathrm{Ker}}

\newcommand{\Ext}{\mathrm{Ext}}

\newcommand{\Inj}{\mathrm{ Inj}}

\newcommand{\Cogen}{\mathrm{Cogen}}

\newcommand{\Modr}{\mathrm{ Mod}\text{-}}

\begin{document}


\title{A Note on Cosilting Modules}

\author{Flaviu Pop}
\thanks{}

\address{Flaviu Pop: "Babe\c s-Bolyai" University, Faculty of Economics and Business Administration, str. T. Mihali, nr. 58-60, 400591, Cluj-Napoca, Romania}

\email{flaviu.v@gmail.com; flaviu.pop@econ.ubbcluj.ro}

\date{\today}

\subjclass[2010]{}


\begin{abstract} The notion of cosilting module was recently introduced as a generalization of the notion of cotilting modules. In this paper, we give a characterization of (partial) cosilting modules in terms of two-term cosilting complexes. Moreover, we show that to every cosilting module could be associated a particular $t$-structure in the derived module category.
\end{abstract}

\maketitle

\section{Introduction} Tilting modules and tilting complexes are very important in the representation theory in order to compare different categories of modules or derived categories. Silting complexes, introduced by Keller and Vossieck \cite{KellerVossieck_1988}, are also important tools in order to study $t$-structures of the bounded derived categories, and they are a generalization of tilting complexes. This topic is intensively studied by many authors (see, for instance, \cite{AiharaIyama}, \cite{IyamaJorgensenYang_2014}, \cite{KoenigYang_2014}, \cite{PsaroudakisVitoria} and \cite{Wei_2013}). In \cite{AngeleriMarksVitoria_2015}, the authors introduced a concept of silting module as a common generalization of tilting modules over an arbitrary ring and support $\tau$-tilting modules over a finite dimensional algebra (see \cite{AdachiIyamaReiten}) and they study important properties of these modules. More exactly, they show that silting modules generate torsion classes that provide left approximations and every partial silting module admits an analog of the Bongartz complement.  Moreover, they also established a connection of these modules with (two-term) silting complexes and with certain $t$-structures and co-$t$-structures in the derived module category. In \cite{AngeleriMarksVitoria_preprint}, the authors give a relation between (partial) silting modules and ring epimorphisms and, they study in details the case of hereditary rings. In \cite{MarksStovicek}, it is shown that, for finite dimensional algebras of finite representation type, silting modules are in bijection with universal localizations. Recently, in \cite{AngeleriHrbek}, the authors give a classification of silting modules over commutative rings, establishing a bijective correspondence of them with Gabriel filters of finite type. The dual notion, of cosilting module, was independently introduced in \cite{BreazPop_preprint} and \cite{ZhangWei}, as a generalization of the concept of cotilting module. Following the idea from \cite{AngeleriMarksVitoria_2015}, where the authors characterize silting modules by the corresponding two-term silting complexes, in this paper we continue the work of \cite{BreazPop_preprint} and we give a characterization of (partial) cosilting modules in terms of two-term cosilting complexes (Proposition \ref{partial_cosilting} and Corollary \ref{cosilting_complex}). Moreover, we show that to every cosilting module could be associated a particular $t$-structure in the derived module category (Proposition \ref{cosilting_t_structure}).

Throughout this paper, by a ring $R$ we will understand a unital associative ring, an $R$-module is a right $R$-module and we will denote by $\mathrm{Mod}(R)$ the category of all right $R$-modules. If $R$ is a ring, we denote by $\Inj(R)$ the full subcategory of $\mathrm{Mod}(R)$ consisting in all injective $R$-modules. An $R$-module $X$ is said to be $T$-cogenerated if it can be embedded into a direct product of copies of $T$ and we denote by $\Cogen(T)$ the class of all $T$-cogenerated $R$-modules. For an $R$-module $Y$, we consider the following classes, denoted by $^{\circ}Y$ and $^{\perp}Y$, defined as follows:$$^{\circ}Y = \{X \in \mathrm{Mod}(R) | \Hom_{R}(X,Y) = 0\}$$ and $$^{\perp}Y = \{X \in \mathrm{Mod}(R) | \Ext^{1}_{R}(X,Y) = 0\}.$$

The unbounded homotopy (respectively, derived) category of $\mathrm{Mod}(R)$ will be denoted by $\rmK(R)$ (respectively, $\rmD(R)$), while the bounded homotopy (respectively, derived) category of $\mathrm{Mod}(R)$ will be denoted by $\rmK^{b}(R)$ (respectively, $\rmD^{b}(R)$). If $Y^{\bullet}$ is an object in $\rmD(R)$ and $n$ is an integer, we set the following orthogonal classes associated to $Y^{\bullet}$: $$^{\perp_{\geq n}}Y^{\bullet} = \{X^{\bullet} \in \mathrm{D}(R) | \Hom_{\mathrm{D}(R)}(X^{\bullet},Y^{\bullet}[i]) = 0\text{, for all } i \geq n \}$$respectively$$^{\perp_{\leq n}}Y^{\bullet} = \{X^{\bullet} \in \mathrm{D}(R) | \Hom_{\mathrm{D}(R)}(X^{\bullet},Y^{\bullet}[i]) = 0\text{, for all } i \leq n \}.$$
If $X$ is an $R$-module, then $X$ can be viewed as a complex concentrated in 0-th degree and it will be denoted by $X^{\bullet}$. If $\zeta : Q_{0} \to Q_{1}$ is an $R$-homomorphism, then it can be also viewed as a complex concentrated in degrees 0 and 1, and thus it will be denoted by $\zeta^{\bullet}$. If $X^{\bullet} = (X_{j},d_{j})$ is a complex of $R$-modules, i.e.$$X^{\bullet}: \dots \overset{d_{i-2}}\longrightarrow X_{i-1}\overset{d_{i-1}}\longrightarrow X_{i}\overset{d_{i}}\longrightarrow X_{i+1}\overset{d_{i+1}}\longrightarrow \dots$$ then we will denote by $H^{i}(X^{\bullet})$ the i-th cohomology of $X^{\bullet}$, i.e. $H^{i}(X^{\bullet})=\Ker(d_{i})/\Ima(d_{i-1})$, for some integer $i$.

\section{Cosilting Modules as two-term complexes}

In order to define the notion of (partial) cosilting module, we need to set the following class, defined as in \cite{BreazPop_preprint} and \cite{ZhangWei}. If $\zeta: Q_{0} \to Q_{1}$ is an $R$-homomorphism, then we define the class $\CB_{\zeta}$ as follows $$\CB_{\zeta}=\{X \in \Modr R \mid \Hom_{R}(X,\zeta) \text{ is an epimorphism} \}.$$

In the following lemma it is established some closure properties of the class $\CB_{\zeta}$.

\begin{lemma}\label{closure_prop}\cite[Lemma 2.3]{BreazPop_preprint} Let $\zeta: Q_{0} \to Q_{1}$ be an $R$-homomorphism. The following assertions hold.
\begin{itemize}
   \item[(1)] The class $\CB_{\zeta}$ is closed under direct sums.
   \item[(2)] If $Q_{1}$ is injective then the class $\CB_{\zeta}$ is closed under submodules.
   \item[(3)] If $Q_{0}$ is injective then the class $\CB_{\zeta}$ is closed under extensions.
\end{itemize}
\end{lemma}

We recall from \cite{ColpiTonoloTrlifaj_1997} the definition of the notion of (partial) cotilting module. A right $R$-module $T$ is {\sl partial cotilting} if and only if $\Cogen(T) \subseteq {^{\perp}T}$ and the class ${^{\perp}T}$ is a torsion free class. A right $R$-module $T$ is {\sl cotilting} if and only if $\Cogen(T)={^{\perp}T}$. Now we define the notion of {\sl (partial) cosilting module} as a generalization of the notion of (partial) cotilting module.

\begin{definition} We say that an $R$-module $T$ is:
   \begin{itemize}
      \item[(1)] {\sl partial cosilting} ({\sl with respect to $\zeta$}), if there exists an injective copresentation of $T$ $$0 \to T \overset{f}\longrightarrow Q_{0} \overset{\zeta}\longrightarrow Q_{1}$$such that:
                 \begin{itemize}
                    \item[(a)] $T\in\CB_{\zeta}$, and
                    \item[(b)] the class $\CB_{\zeta}$ is closed under direct products;
                 \end{itemize}
	  \item[(2)] {\sl cosilting} ({\sl with respect to $\zeta$}), if there exists an injective copresentation $$0 \to T \overset{f}\longrightarrow Q_{0} \overset{\zeta}\longrightarrow Q_{1}$$ of $T$ such that $\Cogen(T)=\CB_{\zeta}$.
   \end{itemize}
\end{definition}

\begin{remark} Let $T$ be a right $R$-module. Then $T$ is (partial) cotilting if and only if $T$ is (partial) cosilting with respect to an epimorphic injective copresentation for $T$. We mention that there are many examples of cosilting modules, that are not cotilting modules (see, for example, \cite[Example 3.3(c)]{BreazPop_preprint}).

In \cite[Proposition 4.1.(c)]{Angeleri_preprint_2016} it is shown that the dual of a silting module is a cosilting module. More exactly, let $k$ be a commutative ring and let $A$ be a $k$-algebra. If $M$ is an $A$-module, we denote by $M^{+}$ the dual of $M$ with respect to an injective cogenerator of $\mathrm{Mod}(k)$. Let $\sigma:P_{-1} \to P_{0}$ be an $A$-homomorphism between projective $A$-modules. If $T$ is a silting module with respect to $\sigma$, then $T^{+}$ is a cosilting module with respect to $\sigma^{+}$.

We note that not all cosilting modules arise as duals of silting modules (as in the tilting-cotilting case).

In \cite{BreazZemlicka}, it is shown that, for general rings, the covering torsion free classes of modules are exactly the classes of the form $\Cogen(T)$, where $T$ is a cosilting module.
\end{remark}

\begin{lemma}\label{inclusions}\cite[Lemma 3.4]{BreazPop_preprint} If the $R$-module $T$ is partial cosilting with respect to the injective copresentation $\zeta : Q_{0} \to Q_{1}$, then $\Cogen(T) \subseteq \CB_{\zeta} \subseteq {^{\perp}T}$.
\end{lemma}

\begin{lemma}\label{homotopy} Let $d_{0}:X_{0} \to X_{1}$ be a homomorphism with the kernel $K$ and let $\zeta: Q_{0} \to Q_{1}$ be a homomorphism in $\Inj(R)$. Then $K\in\CB_{\zeta}$ if and only if for every homomorphism $f\in\Hom_{R}(X_{0},Q_{1})$ there are $s_{0}:X_{0} \to Q_{0}$ and $s_{1}:X_{1} \to Q_{1}$ such that $f=s_{1}d_{0}+\zeta s_{0}$.
\end{lemma}
\begin{proof} Assume that $K \in \CB_{\zeta}$ and let $f\in\Hom_{R}(X_{0},Q_{1})$. Let $\sigma:K \to X_{0}$ be the inclusion. Since $f \sigma \in \Hom_{R}(K,Q_{1})$ and $\Hom_{R}(K,\zeta)$ is an epimorphism, there is $h \in \Hom_{R}(K,Q_{0})$ such that $\zeta h = f \sigma$. By the injectivity of $Q_{0}$, there is $s_{0} : X_{0} \to Q_{0}$ such that $h = s_{0} \sigma$. We have that $(f - \zeta s_{0}) \sigma = 0$, hence $K = \Ima(\sigma) \subseteq \Ker(f-\zeta s_{0})$. Then we can define the morphism $\overline{f - \zeta s_{0}} : X_{0}/K \to Q_{1}$, induced by $f - \zeta s_{0}$. Also, the morphism $\overline{d_{0}} : X_{0}/K \to X_{1}$, induced by $d_{0}$, is in fact a monomorphism. Since $Q_{1}$ is injective, there is $s_{1} : X_{1} \to Q_{1}$ such that $\overline{f - \zeta s_{0}} = s_{1} \overline{d_{0}}$. Therefore $f = s_{1}d_{0} + \zeta s_{0}$.

\[\xymatrix{0\ar[r] &K\ar[r]^{\sigma}\ar@{-->}[d]_{h}  &X_{0}\ar[r]^{d_{0}}\ar[d]^{f}\ar@{-->}[dl]^{s_{0}} &X_{1}\ar@{-->}[dl]^{s_{1}}\\
&Q_{0}\ar[r]^{\zeta} &Q_{1}\\
}\]

Conversely, let $g\in\Hom_{R}(K,Q_{1})$. Since $Q_{1}$ is injective, there is $f:X_{0} \to Q_{1}$ such that $g=f\sigma$. From hypothesis, there are $s_{0}:X_{0} \to Q_{0}$ and $s_{1}:X_{1} \to Q_{1}$ such that $f=s_{1}d_{0} + \zeta s_{0}$. Then $\Hom_{R}(K,\zeta)(s_{0}\sigma) = \zeta s_{0}\sigma = (f-s_{1}d_{0})\sigma = f\sigma - s_{1}d_{0}\sigma = f\sigma = g$. It follows that $\Hom_{R}(K,\zeta)$ is an epimorphism, hence $K\in\CB_{\zeta}$.

\end{proof}

Now we define the notion of cosilting complex, firstly introduced by Zhang and Wei in \cite{ZhangWei}. Before stating this definition we recall that, if $X^{\bullet}$ is an object in $\rmD(R)$, then $\mathrm{Adp}_{\rmD(R)}(X^{\bullet})$ denotes the class of complexes isomorphic in the derived category $\rmD(R)$ to a direct summand of some direct product of copies of $X^{\bullet}$. Moreover, we say that an object $Y^{\bullet}$ of $\rmD(R)$ {\sl cogenerates} $\rmD(R)$, if whenever an object $X^{\bullet}$ in $\rmD(R)$ with $\Hom_{\rmD(R)}(X^{\bullet},Y^{\bullet}[i])=0$ for all $i\in\BBZ$, then $X^{\bullet}=0$ in $\rmD(R)$.

\begin{definition} A complex $T^{\bullet}$ is said to be {\sl cosilting} if it satisfies the following conditions:
\begin{itemize}
    \item[(1)] $T^{\bullet} \in \rmK^{b}(\Inj(R))$;
    \item[(2)] $T^{\bullet}$ is prod-semi-selforthogonal, i.e. $\Hom_{\rmD(R)}(T^{\bullet I},T^{\bullet}[i]) = 0$, for all sets $I$ and for all $i>0$;
    \item[(3)] $T^{\bullet}$ cogenerates $\rmD(R)$.
\end{itemize}
A complex $T^{\bullet}$ which satisfies the conditions (1) and (2) from above definition is said to be {\sl partial cosilting}.
\end{definition}

We mention that we can freely interchange $\rmD(R)$ with $\rmK(R)$ from the above definitions since, for all $X^{\bullet}\in\rmK(R)$ and for all $Y^{\bullet}\in\rmK^{b}(\Inj(R))$, we have the equality $\Hom_{\rmD(R)}(X^{\bullet},Y^{\bullet}) = \Hom_{\rmK(R)}(X^{\bullet},Y^{\bullet})$ (for more details, see \cite{Keller_1998}).

Also, we mention that in \cite{ZhangWei} a complex $T^{\bullet}$ is said to be a {\sl cosilting} complex if it satisfies the conditions (1) and (2) from the above definition and moreover the following condition:
\begin{itemize}
   \item [(3$^{\prime}$)] $\rmK^{b}(\Inj(R))$ coincides with the smallest triangulated subcategory containing $\mathrm{Adp}_{\rmD(R)}(T^{\bullet})$.
\end{itemize}
holds. We note that the condition (3$^{\prime}$) from the definition of cosilting complex, given by Zhang and Wei, implies the condition (3) from the definition above, i.e. $T^{\bullet}$ cogenerates $\rmD(R)$.

We denote by $\rmD^{\leq 0}$ (respectively, by $\rmD^{\geq 0}$) the subcategory of complexes with cohomologies lying in non-positive (respectively, non-negative) degrees.

\begin{lemma}\label{zero_homology} Let $\zeta:Q_{0} \to Q_{1}$ be a homomorphism in $\Inj(R)$ with $T=\Ker(\zeta)$. Then the following assertions hold:
\begin{itemize}
   \item[(a)] If $X^{\bullet} \in \rmD^{\geq 0}$, then $X^{\bullet} \in {^{\perp_{> 0}}\zeta^{\bullet}}$ if and only if $H^{0}(X^{\bullet}) \in \CB_{\zeta}$.
   \item[(b)] If $X^{\bullet} \in \rmD^{\leq 0}$, then $X^{\bullet} \in {^{\perp_{\leq 0}}\zeta^{\bullet}}$ if and only if $H^{0}(X^{\bullet}) \in {^{\circ}T}$.
\end{itemize}
\end{lemma}
\begin{proof}(a) Let $X^{\bullet} \in \rmD^{\geq 0}$. Then we may assume, without loss of generality, that  $$X^{\bullet} : \dots \longrightarrow 0 \longrightarrow 0 \longrightarrow X_{0}\overset{d_{0}}\longrightarrow X_{1}\overset{d_{1}}\longrightarrow  X_{2}\overset{d_{2}}\longrightarrow \dots$$We denote by $K$ the cohomology in the zero degree of $X^{\bullet}$, i.e. $K = H^{0}(X^{\bullet}) = \Ker(d_{0})$, and consider $\sigma : K \to X_{0}$ to be the inclusion.

Assume that $X^{\bullet} \in {^{\perp_{> 0}}\zeta^{\bullet}}$. Since $\Hom_{\rmD(R)}(X^{\bullet},\zeta^{\bullet}[1]) = 0$ it follows, by applying Lemma \ref{homotopy}, that $K\in\CB_{\zeta}$.

Conversely, assume that $K \in \CB_{\zeta}$. Since $\zeta^{\bullet}$ is a two-term complex, concentrated in degrees 0 and 1, it is obvious that $\Hom_{\rmK(R)} (X^{\bullet},\zeta^{\bullet}[i]) = 0$, for all $i \geq 2$. Applying Lemma \ref{homotopy}, we obtain $\Hom_{\rmK(R)} (X^{\bullet},\zeta^{\bullet}[1]) = 0$. Therefore $X^{\bullet} \in {^{\perp_{> 0}}\zeta^{\bullet}}$.

(b) Let $X^{\bullet} \in \rmD^{\leq 0}$. Then we may assume, without loss of generality, that $$X^{\bullet}: \dots \longrightarrow X_{-2} \overset{d_{-2}}\longrightarrow X_{-1} \overset{d_{-1}}\longrightarrow X_{0} \longrightarrow 0 \longrightarrow 0 \longrightarrow\dots$$Hence $H^{0}(X^{\bullet}) = X_{0}/\Ima(d_{-1})$.

Assume that $X^{\bullet} \in {^{\perp_{\leq 0}}\zeta^{\bullet}}$. Let $f \in \Hom_{R}(X_{0}/\Ima(d_{-1}), T)$. Since the composition $\sigma f \pi$ lies in $\Hom_{R}(X_{0}, Q_{0})$, where $\sigma : T \to Q_{0}$ is the inclusion and $\pi:X_{0} \to X_{0}/\Ima(d_{-1})$ is the canonical epimorphism and taking into account that $(\sigma f \pi) d_{-1} = 0 = \zeta (\sigma f \pi)$, we have that $(...,0,0,\sigma f \pi,0,0,...)$ is a chain map from $X^{\bullet}$ to $\zeta^{\bullet}$ and we will denote it by $f^{\bullet}$. By hypothesis, $f^{\bullet}$ vanishes in $\rmK(R)$, hence $\sigma f \pi = 0$, so that $f = 0$. It follows that $\Hom_{R}(X_{0}/\Ima(d_{-1}), T) = 0$.

Conversely, suppose that $H^{0}(X^{\bullet}) \in {^{\circ}T}$. Is obvious that $\Hom_{\rmK(R)}(X^{\bullet},\zeta^{\bullet}[i]) = 0$, for all $i<0$.  Let $f^{\bullet} \in \Hom_{\rmK(R)}(X^{\bullet},\zeta^{\bullet})$. Since $f^{\bullet} = (\dots, 0, 0, f, 0, 0, \dots)$ is a chain map, hence $fd_{-1} = 0 = \zeta f$, we have $\Ima(d_{-1}) \subseteq \Ker(f)$ and $\Ima(f) \subseteq \Ker(\zeta)$. It follows that the induced map by $f$, i.e. $\overline{f} : X_{0}/\Ima(d_{-1}) \to \Ker(\zeta)$, is well-defined. By assumption, $\overline{f} = 0$ in $\mathrm{Mod}(R)$, hence $f = 0$ in $\mathrm{Mod}(R)$, so that $f^{\bullet} = 0$ in $\rmK(R)$. Therefore $X^{\bullet} \in {^{\perp_{\leq 0}}\zeta^{\bullet}}$.
\end{proof}

\begin{proposition}\label{partial_cosilting} Let $\zeta:Q_{0} \to Q_{1}$ be a homomorphism in $\Inj(R)$ with $T=\Ker(\zeta)$. The following statements are equivalent:
\begin{itemize}
   \item[(a)] $T$ is a partial cosilting $R$-module with respect to $\zeta$;
   \item[(b)] \begin{itemize}
                 \item[(i)] $\Hom_{\rmD(R)}(\zeta^{\bullet I},\zeta^{\bullet}[i]) = 0$, for all sets $I$ and for all integers $i>0$;
                 \item[(ii)] The class ${^{\perp_{>0}}\zeta^{\bullet}} \bigcap \rmD^{\geq 0}$ is closed under direct products.
              \end{itemize}
   \item[(c)] \begin{itemize}
                 \item[(i)] $\zeta^{\bullet}$ is a partial cosilting complex;
                 \item[(ii)] The class ${^{\perp_{>0}}\zeta^{\bullet}} \bigcap \rmD^{\geq 0}$ is closed under direct products.
              \end{itemize}
\end{itemize}
\end{proposition}
\begin{proof}(a)$\Rightarrow$(b) Suppose that $T$ is partial cosilting with respect to $\zeta$. Then $T \in \CB_{\zeta}$ and the class $\CB_{\zeta}$ is closed under direct products. If $I$ is a set, then is obvious that $\Hom_{\rmK(R)}(\zeta^{\bullet I},\zeta^{\bullet}[i]) = 0$, for all $i \geq 2$. Since $T^{I} \in \CB_{\zeta}$, it follows, by applying Lemma \ref{homotopy}, that $\Hom_{\rmK(R)}(\zeta^{\bullet I},\zeta^{\bullet}[1]) = 0$.

Now, let $X^{\bullet}_{\alpha}$ be a set-indexed family of objects in ${^{\perp_{>0}}\zeta^{\bullet}} \bigcap \rmD^{\geq 0}$. Is obvious that $\prod_{\alpha} X^{\bullet}_{\alpha} \in \rmD^{\geq 0}$. Since $X^{\bullet}_{\alpha} \in {^{\perp_{>0}}\zeta^{\bullet}}$, it follows, by Lemma \ref{zero_homology}(a), that $H^{0}(X^{\bullet}_{\alpha}) \in \CB_{\zeta}$ and, since $\CB_{\zeta}$ is closed under direct products, we have that $\prod_{\alpha} H^{0}(X^{\bullet}_{\alpha}) \in \CB_{\zeta}$. Since $H^{0}$ commutes with direct products, we obtain $H^{0}(\prod_{\alpha} X^{\bullet}_{\alpha}) \in \CB_{\zeta}$, hence, by Lemma \ref{zero_homology}(a), $\prod_{\alpha} X^{\bullet}_{\alpha} \in {^{\perp_{>0}}\zeta^{\bullet}}$.

(b)$\Rightarrow$(a) Let $h\in\Hom_{R}(T,Q_{1})$. Since $Q_{1}$ is injective, there is a morphism $f \in \Hom_{R}(Q_{0},Q_{1})$ such that $h = f \sigma$, where $\sigma:T \to Q_{0}$ is the inclusion. If we denote the chain map $(\dots,0,0,f,0,0,\dots)$ by $f^{\bullet}$, so that $f^{\bullet} \in \Hom_{\rmK(R)}(\zeta^{\bullet}, \zeta^{\bullet}[1])$, we have by $(i)$ that $f^{\bullet} = 0$ in the homotopy category $\rmK(R)$, hence there are the morphisms $s_{0} : Q_{0} \to Q_{0}$ and $s_{1} : Q_{1} \to Q_{1}$ such that $f = s_{1}\zeta + \zeta s_{0}$. Since $h = f \sigma = \zeta s_{0} \sigma$, we obtain $\Hom_{R}(T,\zeta)(s_{0} \sigma) = h$. Therefore $T\in\CB_{\zeta}$.

Let $X_{\alpha}\in\CB_{\zeta}$, be a set-indexed family of $R$-modules. Since $\Hom_{R}(X_{\alpha},\zeta)$ is an epimorphism, it follows that $X_{\alpha}^{\bullet} \in {^{\perp_{>0}}\zeta^{\bullet}}$, hence $X_{\alpha}^{\bullet} \in {^{\perp_{>0}}\zeta^{\bullet}} \bigcap \rmD^{\geq 0}$. By (ii), we obtain that $\prod_{\alpha} X_{\alpha}^{\bullet} \in {^{\perp_{>0}}\zeta^{\bullet}} \bigcap \rmD^{\geq 0}$, so that $\prod_{\alpha} X_{\alpha}^{\bullet} \in {^{\perp_{>0}}\zeta^{\bullet}}$. By Lemma \ref{zero_homology}(a), we have $\prod_{\alpha} X_{\alpha}\in\CB_{\zeta}$. Therefore the class $\CB_{\zeta}$ is closed under direct products.

(b)$\Leftrightarrow$(c) It is obvious.
\end{proof}

For a complex $X^{\bullet}=(X_{i},d_{i})$ of $R$-modules, $$X^{\bullet}:\dots\overset{d_{n-3}}\longrightarrow X_{n-2}\overset{d_{n-2}}\longrightarrow X_{n-1}\overset{d_{n-1}}\longrightarrow X_{n}\overset{d_{n}}\longrightarrow X_{n+1}\overset{d_{n+1}}\longrightarrow X_{n+2}\overset{d_{n+2}}\longrightarrow \dots,$$we define the following truncations:$$\tau^{\leq n}(X^{\bullet}): \dots\overset{d_{n-3}}\longrightarrow X_{n-2}\overset{d_{n-2}}\longrightarrow X_{n-1}\overset{d_{n-1}}\longrightarrow \Ker(d_{n})\overset{}\longrightarrow 0\overset{}\longrightarrow 0\overset{}\longrightarrow \dots,$$ and $$\tau^{\geq n}(X^{\bullet}):\dots\overset{}\longrightarrow 0\overset{}\longrightarrow 0\overset{}\longrightarrow X_{n}/\Ima(d_{n-1})\overset{\overline{d_{n}}}\longrightarrow X_{n+1}\overset{d_{n+1}}\longrightarrow X_{n+2}\overset{d_{n+2}}\longrightarrow \dots.$$

\begin{lemma}\label{truncations} Let $\zeta : Q_{0} \to Q_{1}$ be a homomorphism in $\Inj(R)$ and let $X^{\bullet}$ be an object in $\rmK(R)$. Then:
\begin{itemize}
    \item[(a)] If $\Hom_{\rmK(R)}(X^{\bullet},\zeta^{\bullet}) = 0$, then $\tau^{\leq 0}(X^{\bullet})\in {^{\perp_{\leq 0}}\zeta^{\bullet}}$;
    \item[(b)] If $\Hom_{\rmK(R)}(X^{\bullet},\zeta^{\bullet}[1]) = 0$, then $\tau^{\geq 0}(X^{\bullet})\in {^{\perp_{>0}}\zeta^{\bullet}}$.
\end{itemize}
\end{lemma}
\begin{proof}(a) Since $\zeta^{\bullet}$ is a two-term complex, concentrated in degrees 0 and 1, is obvious that $\Hom_{\rmD(R)}(\tau^{\leq 0}(X^{\bullet}), \zeta^{\bullet}[i]) = 0$, for all $i \leq -1$. Now, consider the triangle induced by the standard $t$-structure $$\tau^{\leq 0}(X^{\bullet}) \longrightarrow X^{\bullet} \longrightarrow \tau^{\geq 1}(X^{\bullet}) \longrightarrow \tau^{\leq 0}(X^{\bullet})[1]$$and, by applying the contravariant functor $\Hom_{\rmD(R)}(-,\zeta^{\bullet})$, we obtain the exact sequence $$\Hom_{\rmD(R)}(X^{\bullet},\zeta^{\bullet}) \to \Hom_{\rmD(R)}(\tau^{\leq 0}(X^{\bullet}),\zeta^{\bullet}) \to \Hom_{\rmD(R)}(\tau^{\geq 1}(X^{\bullet})[-1],\zeta^{\bullet}).$$ From hypothesis we have that $\Hom_{\rmD(R)}(X^{\bullet},\zeta^{\bullet})=0$ and, since $\zeta^{\bullet}$ is a two-term complex in degrees 0 and 1, we also have that $\Hom_{\rmD(R)}(\tau^{\geq 1}(X^{\bullet})[-1],\zeta^{\bullet}) = 0$. Thus, $\Hom_{\rmD(R)}(\tau^{\leq 0}(X^{\bullet}),\zeta^{\bullet})=0$.

(b) It is obvious that $\Hom_{\rmD(R)}(\tau^{\geq 0}(X^{\bullet}), \zeta^{\bullet}[i]) = 0$, for all $i \geq 2$,  since $\zeta^{\bullet}$ is two-term complex concentrated in degrees 0 and 1. Considering the triangle $$\tau^{\leq -1}(X^{\bullet}) \longrightarrow X^{\bullet} \longrightarrow \tau^{\geq 0}(X^{\bullet}) \longrightarrow \tau^{\leq -1}(X^{\bullet})[1]$$and applying the contravariant functor $\Hom_{\rmD(R)}(-,\zeta^{\bullet})$, we get the exact sequence $$\Hom_{\rmD(R)}(\tau^{\leq -1}(X^{\bullet}),\zeta^{\bullet}) \to \Hom_{\rmD(R)}(\tau^{\geq 0}(X^{\bullet})[-1],\zeta^{\bullet}) \to \Hom_{\rmD(R)}(X^{\bullet}[-1],\zeta^{\bullet}).$$Since $\zeta^{\bullet}$ is a two-term complex concentrated in degrees 0 and 1, we have that $\Hom_{\rmD(R)}(\tau^{\leq -1}(X^{\bullet}),\zeta^{\bullet})=0$. By hypothesis, $\Hom_{\rmD(R)}(X^{\bullet}[-1],\zeta^{\bullet})=0$. Therefore $\Hom_{\rmD(R)}(\tau^{\geq 0}(X^{\bullet})[-1],\zeta^{\bullet})=0$, i.e. $\Hom_{\rmD(R)}(\tau^{\geq 0}(X^{\bullet}),\zeta^{\bullet}[1])=0$.
\end{proof}

We recall that an object $Y^{\bullet}$ of $\rmD(R)$ {\sl cogenerates} $\rmD(R)$, if whenever an object $X^{\bullet}$ in $\rmD(R)$ with $\Hom_{\rmD(R)}(X^{\bullet},Y^{\bullet}[i])=0$ for all $i\in\BBZ$, then $X^{\bullet}=0$ in $\rmD(R)$.

\begin{lemma}\label{cogenerates} Let $\zeta : Q_{0} \to Q_{1}$ be a homomorphism in $\Inj(R)$. Then the following are equivalent:
\begin{itemize}
    \item[(a)] $\zeta^{\bullet}$ cogenerates $\rmD(R)$;
    \item[(b)] ${^{\perp_{>0}}\zeta^{\bullet}} \subseteq {\rmD^{\geq 0}}$;
    \item[(c)] ${^{\perp_{\leq 0}}\zeta^{\bullet}} \subseteq {\rmD^{\leq 0}}$.
\end{itemize}
\end{lemma}
\begin{proof}$(a)\Rightarrow(b)$ Suppose that $\zeta^{\bullet}$ cogenerates $\rmD(R)$. Let $X^{\bullet}\in{^{\perp_{>0}}\zeta^{\bullet}}$. From the fact that $\zeta^{\bullet}$ is a two-term complex, concentrated in degrees 0 and 1, it is obvious that $\Hom_{\rmD(R)}(\tau^{\leq -1}(X^{\bullet}),\zeta^{\bullet}[i])=0$, for all $i \leq 0$. Also, since $\Hom_{\rmD(R)}(X^{\bullet},\zeta^{\bullet}[i])=0$, for all $i>0$, is easy to see that $\Hom_{\rmD(R)}(\tau^{\leq -1}(X^{\bullet}), \zeta^{\bullet}[i]) = 0$, for all $i\geq 3$. Applying the functor $\Hom_{\rmD(R)}(-,\zeta^{\bullet})$ to the triangle $$\tau^{\leq -1}(X^{\bullet}) \longrightarrow X^{\bullet} \longrightarrow \tau^{\geq 0}(X^{\bullet}) \longrightarrow \tau^{\leq -1}(X^{\bullet})[1]$$ we obtain that $\Hom_{\rmD(R)}(\tau^{\leq -1}(X^{\bullet}),\zeta^{\bullet}[1]) = 0 = \Hom_{\rmD(R)}(\tau^{\leq -1}(X^{\bullet}),\zeta^{\bullet}[2])$.

In conclusion, $\Hom_{\rmD(R)}(\tau^{\leq -1}(X^{\bullet}), \zeta^{\bullet}[i]) = 0$, for all $i\in\BBZ$. Since $\zeta^{\bullet}$ cogenerates $\rmD(R)$, we have $\tau^{\leq -1}(X^{\bullet}) = 0$ in $\rmD(R)$, hence $\tau^{\leq -1}(X^{\bullet})$ is an exact sequence in $\mathrm{Mod}(R)$. Therefore $X^{\bullet}\in\rmD^{\geq 0}$.

$(b)\Rightarrow(a)$ Let $X^{\bullet}\in\rmD(R)$ such that $\Hom_{\rmD(R)}(X^{\bullet}, \zeta^{\bullet}[i]) = 0$, for all $i\in\BBZ$. If $j$ is an arbitrary integer, we have that  $X^{\bullet}[j] \in {^{\perp_{>0}}\zeta^{\bullet}}$. By hypothesis, $X^{\bullet}[j]\in\rmD^{\geq 0}$, so that $\Ima(d_{i})=\Ker(d_{i+1})$, for all $i \leq j-2$. Therefore $X^{\bullet}$ is an acyclic complex, i.e. $X^{\bullet} = 0$ in $\rmD(R)$.

$(a)\Rightarrow(c)$ Suppose that $\zeta^{\bullet}$ cogenerates $\rmD(R)$. Let $X^{\bullet} \in {^{\perp_{\leq 0}}\zeta^{\bullet}}$. It is obvious that $\Hom_{\rmD(R)}(\tau^{\geq 1}(X^{\bullet}),\zeta^{\bullet}[i])=0$, for all $i \ne -1,0$. We also have that $\Hom_{\rmD(R)}(\tau^{\geq 1}(X^{\bullet}),\zeta^{\bullet}[-1]) = 0 = \Hom_{\rmD(R)}(\tau^{\geq 1}(X^{\bullet}),\zeta^{\bullet})$, by applying the functor $\Hom_{\rmD(R)}(-,\zeta^{\bullet})$ to the triangle $$\tau^{\leq 0}(X^{\bullet}) \longrightarrow X^{\bullet} \longrightarrow \tau^{\geq 1}(X^{\bullet}) \longrightarrow \tau^{\leq 0}(X^{\bullet})[1].$$

Since $\zeta^{\bullet}$ cogenerates $\rmD(R)$, it follows that $\tau^{\geq 1}(X^{\bullet})=0$ in $\rmD(R)$, hence $\tau^{\geq 1}(X^{\bullet})$ is an exact sequence in $\mathrm{Mod}(R)$ and therefore $X^{\bullet} \in \rmD^{\leq 0}$.

$(c)\Rightarrow(a)$ Let $X^{\bullet}\in\rmD(R)$ such that $\Hom_{\rmD(R)}(X^{\bullet}, \zeta^{\bullet}[i]) = 0$, for all $i\in\BBZ$. If $j$ is an arbitrary integer, we have that  $X^{\bullet}[j] \in {^{\perp_{\leq 0}}\zeta^{\bullet}}$. By hypothesis, $X^{\bullet}[j]\in\rmD^{\leq 0}$, so that $\Ima(d_{i})=\Ker(d_{i+1})$, for all $i \geq j$. Therefore $X^{\bullet}$ is an acyclic complex, i.e. $X^{\bullet} = 0$ in $\rmD(R)$.
\end{proof}

As a consequence of the previous result, we obtain the following corollary.

\begin{corollary}\label{ZhangWei} Let $\zeta : Q_{0} \to Q_{1}$ be an homomorphism in $\Inj(R)$ such that $\zeta^{\bullet}$ is a partial cosilting complex. Then $\zeta^{\bullet}$ is a cosilting complex if and only if ${^{\perp_{> 0}}\zeta^{\bullet}} \subseteq \rmD^{\geq 0}$ if and only if ${^{\perp_{\leq 0}}\zeta^{\bullet}}\subseteq\rmD^{\leq 0}$.
\end{corollary}

Now, we give the main result of the paper.

\begin{theorem}\label{cosilting} Let $\zeta:Q_{0} \to Q_{1}$ be a homomorphism in $\Inj(R)$ with $T=\Ker(\zeta)$. The following statements are equivalent:
\begin{itemize}
   \item[(a)] $T$ is a cosilting $R$-module with respect to $\zeta$;
   \item[(b)] The pair $(^{\circ}T,\CB_{\zeta})$ is a torsion pair;
   \item[(c)] \begin{itemize}
                 \item[(i)] $\Hom_{\rmD(R)}(\zeta^{\bullet I},\zeta^{\bullet}[i])=0$, for all sets $I$ and for all $i>0$;
                 \item[(ii)] The class ${^{\perp_{>0}}\zeta^{\bullet}} \bigcap \rmD^{\geq 0}$ is closed under direct products.
                 \item[(iii)] $\zeta^{\bullet}$ cogenerates $\rmD(R)$.
              \end{itemize}
   \item[(d)] \begin{itemize}
                 \item[(i)] $\Hom_{\rmD(R)}(\zeta^{\bullet I},\zeta^{\bullet}[i])=0$, for all sets $I$ and for all $i>0$;
                 \item[(ii)] The class ${^{\perp_{>0}}\zeta^{\bullet}} \bigcap \rmD^{\geq 0}$ is closed under direct products.
                 \item[(iii)] ${^{\perp_{>0}}\zeta^{\bullet}} \subseteq {\rmD^{\geq 0}}$.
              \end{itemize}
   \item[(e)] \begin{itemize}
                 \item[(i)] $\Hom_{\rmD(R)}(\zeta^{\bullet I},\zeta^{\bullet}[i])=0$, for all sets $I$ and for all $i>0$;
                 \item[(ii)] The class ${^{\perp_{>0}}\zeta^{\bullet}} \bigcap \rmD^{\geq 0}$ is closed under direct products.
                 \item[(iii)] ${^{\perp_{\leq 0}}\zeta^{\bullet}} \subseteq {\rmD^{\leq 0}}$.
              \end{itemize}
\end{itemize}
\end{theorem}
\begin{proof}(a)$\Leftrightarrow$(b) See \cite[Corollary 3.5]{BreazPop_preprint}.

(b)$\Rightarrow$(c) Suppose that $(^{\circ}T,\CB_{\zeta})$ is a torsion pair. Since $T$ is (partial) cosilting with respect to $\zeta$, we have, by Proposition \ref{partial_cosilting}, that (i) and (ii) hold.

Let $X^{\bullet} \in \rmD(R)$ such that $\Hom_{\rmD(R)}(X^{\bullet},\zeta^{\bullet}[i]) = 0$ for all $i \in \BBZ$. By Lemma \ref{truncations} (a), $\tau^{\leq 0}(X^{\bullet}[i]) \in {^{\perp_{\leq 0}}\zeta^{\bullet}}\cap\rmD^{\leq 0}$, and then we have, by Lemma \ref{zero_homology}(b), that $H^{i}(X^{\bullet}) = H^{0}(\tau^{\leq 0}(X^{\bullet}[i]))$ belongs to ${^{\circ}T}$. We also have, by Lemma \ref{truncations} (b), that $\tau^{\geq 0}(X^{\bullet}[i]) \in {^{\perp_{> 0}}\zeta^{\bullet}}\cap\rmD^{\geq 0}$, and then we obtain, by applying Lemma \ref{zero_homology}(a), that $H^{i}(X^{\bullet}) = H^{0}(\tau^{\geq 0}(X^{\bullet}[i]))$ lies in $\CB_{\zeta}$. Then $H^{i}(X^{\bullet}) \in {^{\circ}T} \bigcap \CB_{\zeta}$. From the fact that $(^{\circ}T,\CB_{\zeta})$ is a torsion pair, we obtain that $H^{i}(X^{\bullet}) = 0$, hence $X^{\bullet}$ is acyclic. Therefore $X^{\bullet} = 0$ in $\rmD(R)$.

(c)$\Rightarrow$(a) From (i) and (ii), we have, by Proposition \ref{partial_cosilting}, that $T$ is partial cosilting with respect to $\zeta$, hence, by Lemma \ref{inclusions}, we have $\Cogen(T) \subseteq \CB_{\zeta}$.

Now, let $X \in \CB_{\zeta}$. By \cite[Corollary 3.5]{BreazPop_preprint}, we have that the pair $({^{\circ}T},\Cogen(T))$ is a torsion pair, hence there is a submodule $Y$ of $X$ such that $Y \in {^{\circ}T}$ and $X/Y \in \Cogen(T)$. If we see $Y$ as a complex concentrated in degree zero then it is obvious that $\Hom_{\rmD(R)}(Y^{\bullet},\zeta^{\bullet}[i]) = 0$, for all $i \ne 0,1$. By Lemma \ref{homotopy}, it is easy to deduce that the class $\CB_{\zeta}$ coincides with the class of all $R$-modules $Z$ such that $\Hom_{\rmD(R)}(Z^{\bullet},\zeta^{\bullet}[1])=0$. Since $X$ is in the class $\CB_{\zeta}$, $Y$ is a submodule of $X$ and the class $\CB_{\zeta}$ is closed under submodules, we have that $ \Hom_{\rmD(R)}(Y^{\bullet}, \zeta^{\bullet}[1])=0$. Moreover, we observe that $\Hom_{\rmD(R)}(Z^{\bullet},H^{0}(\zeta^{\bullet})) \cong \Hom_{\rmD(R)}(Z^{\bullet},\zeta^{\bullet})$, for any $R$-module $Z$. Since $Y \in {^{\circ}T}$, i.e. $\Hom_{R}(Y,T)=0$, we have $\Hom_{\rmD(R)}(Y^{\bullet},H^{0}(\zeta^{\bullet}))=0$, hence $\Hom_{\rmD(R)}(Y^{\bullet},\zeta^{\bullet})=0$. It follows that $\Hom_{\rmD(R)}(Y^{\bullet},\zeta^{\bullet}[i]) = 0$, for all $i \in \BBZ$, hence, by (iii), we have that $Y^{\bullet} = 0$ in $\rmD(R)$. It follows that $Y^{\bullet}$ is acyclic, so that $Y = 0$ in $\mathrm{Mod}(R)$. Therefore $X \in \Cogen(T)$.

(c)$\Leftrightarrow$(d)$\Leftrightarrow$(e) It follows by Lemma \ref{cogenerates}.
\end{proof}

\begin{corollary}\label{cosilting_complex} Let $\zeta:Q_{0} \to Q_{1}$ be a homomorphism in $\Inj(R)$ with $T=\Ker(\zeta)$. The following statements are equivalent:
\begin{itemize}
   \item[(a)] $T$ is a cosilting $R$-module with respect to $\zeta$;
   \item[(b)] The pair $(^{\circ}T,\CB_{\zeta})$ is a torsion pair;
   \item[(c)] $\zeta^{\bullet}$ is a cosilting complex;
\end{itemize}
\end{corollary}
\begin{proof} It follows by Theorem \ref{cosilting} and Corollary \ref{ZhangWei}.
\end{proof}

\begin{definition} Let $\CALD$ be a triangulated category. A {\sl $t$-structure} in $\CALD$ is a pair of subcategories $(\CV^{\leq 0}, \CV^{\geq 0})$ such that:
\begin{itemize}
   \item[(1)] $\Hom_{\CALD}(\CV^{\leq 0}, \CV^{\geq 0}[-1])=0$;
   \item[(2)] $\CV^{\leq 0}[1]\subseteq\CV^{\leq 0}$;
   \item[(3)] For every $X$ in $\CALD$, there is a triangle $$Y \to X \to W \to Y[1]$$ such that $Y$ lies in $\CV^{\leq 0}$ and $W$ lies in $\CV^{\geq 0}[-1]$.
\end{itemize}
For a $t$-structure $(\CV^{\leq 0}, \CV^{\geq 0})$, the intersection $\CV^{\leq 0}\cap\CV^{\geq 0}$ is called the {\sl heart} and $\CV^{\leq 0}$ is called the {\sl aisle}.

\end{definition}

\begin{proposition}\cite[Proposition 2.1]{HappelReitenSmalo_1996} If $(\CT,\CF)$ is a torsion pair in an abelian category $\CA$, then the pair $(\CALD^{\leq 0},\CALD^{\geq 0})$ is a t-structure on $\rmD^{b}(\CA)$, where $$\CALD^{\leq 0} = \{ X^{\bullet}\in\rmD^{b}(\CA) | H^{i}(X^{\bullet})=0 \text{ for all } i>0, H^{0}(X^{\bullet})\in\CT \}$$ and $$\CALD^{\geq 0} = \{ X^{\bullet}\in\rmD^{b}(\CA) | H^{i}(X^{\bullet})=0 \text{ for all } i<-1, H^{-1}(X^{\bullet})\in\CF \}.$$
\end{proposition}

Now we have the following result which states that every cosilting module induces a $t$-structure.

\begin{proposition}\label{cosilting_t_structure} Let $\zeta:Q_{0} \to Q_{1}$ be a homomorphism in $\Inj(R)$ with the kernel $T$. If $T$ is a cosilting module with respect to $\zeta$, then the pair $({^{\perp_{\leq 0}}\zeta^{\bullet}},{^{\perp_{>1}}\zeta^{\bullet}})$ is a t-structure on $\rmD^{b}(R)$ (built from the torsion pair $(^{\circ}T,\CB_{\zeta})$).
\end{proposition}
\begin{proof} Assume that $T$ is cosilting with respect to $\zeta$. Then we have, by Theorem \ref{cosilting}, that $\zeta^{\bullet}$ cogenerates $\rmD(R)$ and the pair $(^{\circ}T,\CB_{\zeta})$ is a torsion pair. It follows that the pair $(\overline{\CALD}^{\leq 0},\overline{\CALD}^{\geq 0})$ is a t-structure on $\rmD^{b}(R)$, where the considered classes are defined as follows: $$\overline{\CALD}^{\leq 0} = \{ X^{\bullet}\in\rmD^{\leq 0} | H^{0}(X^{\bullet}) \in {^{\circ}T} \}$$and$$\overline{\CALD}^{\geq 0} = \{ X^{\bullet} \in \rmD^{\geq -1} | H^{-1}(X^{\bullet}) \in \CB_{\zeta} \}.$$

By Lemma \ref{zero_homology}(b), it is easy to see that $\overline{\CALD}^{\leq 0} \subseteq {^{\perp_{\leq 0}}\zeta^{\bullet}}$. If $X^{\bullet}\in{^{\perp_{\leq 0}}\zeta^{\bullet}}$, then, by Lemma \ref{cogenerates}, we have $X^{\bullet}\in\rmD^{\leq 0}$, hence, by Lemma \ref{zero_homology}(b), $X_{0}/\Ima(d_{-1}) = H^{0}(X^{\bullet})\in{^{\circ}T}$. It follows that $X^{\bullet} \in \overline{\CALD}^{\leq 0}$. Therefore $\overline{\CALD}^{\leq 0} = {^{\perp_{\leq 0}}\zeta^{\bullet}}$.

If $X^{\bullet}\in\overline{\CALD}^{\geq 0}$, then $X^{\bullet}[-1]\in\rmD^{\geq 0}$ and $H^{0}(X^{\bullet}[-1])\in\CB_{\zeta}$, hence, by Lemma \ref{zero_homology}(a), we obtain that $X^{\bullet}[-1] \in {^{\perp_{>0}}\zeta^{\bullet}}$, so that $X^{\bullet}\in{^{\perp_{>1}}\zeta^{\bullet}}$. If $X^{\bullet} \in {^{\perp_{>1}}\zeta^{\bullet}}$, then $X^{\bullet}[-1]\in{^{\perp_{>0}}\zeta^{\bullet}}$, hence, by Lemma \ref{cogenerates}, $X^{\bullet}[-1]\in\rmD^{\geq 0}$, so that $X^{\bullet}\in\rmD^{\geq -1}$. By Lemma \ref{zero_homology}(a), we have $\Ker(d_{-1}) = H^{0}(X^{\bullet}[-1])\in\CB_{\zeta}$. Thus $X^{\bullet}\in\overline{\CALD}^{\geq 0}$. Therefore $\overline{\CALD}^{\geq 0} = {^{\perp_{>1}}\zeta^{\bullet}}$.
\end{proof}

\subsection*{Acknowledgement} I would like to thank George Ciprian Modoi for his helpful suggestions and comments that improve this paper. Furthermore, I present my thanks to the referee for his/her suggestions that also improve this article.

\end{document}